\documentclass[runningheads]{llncs}
\usepackage{graphicx}
\usepackage{amssymb, amsmath}

\renewcommand{\i}{\mathrm{i}}
\numberwithin{equation}{section}

\newcommand{\N}{\mathbb{N}}
\begin{document}
\title{Gauge freedom of entropies \\ on $q$-Gaussian measures}
\titlerunning{Gauge freedom of entropies on $q$-Gaussian measures}
\author{Hiroshi Matsuzoe
\thanks{The both authors were supported in part by JSPS Grant-in-Aid for Scientific Research (KAKENHI) 16KT0132.}
\thanks{HM was supported in part by KAKENHI 19K03489.}
\inst{1}
\and
Asuka Takatsu ${}^*$
\thanks{AT was supported in part by KAKENHI 19K03494, 19H01786.}
\inst{2,3}}
\authorrunning{H. Matsuzoe and A. Takatsu}
\institute{ 
Department of Computer Science, Nagoya Institute of Technology, Nagoya, Japan
\email{matsuzoe@nitech.ac.jp} \and
Department of Mathematical Sciences, Tokyo Metropolitan University, \\ Tokyo, Japan
\email{asuka@tmu.ac.jp} \and
RIKEN Center for Advanced Intelligence Project (AIP), Tokyo, Japan}
\maketitle              
\begin{abstract}
A $q$-Gaussian measure is a generalization of a Gaussian measure.
This generalization is obtained by replacing the exponential function with the power function of exponent $1/(1-q)$ ($q\neq 1$).
The limit case $q=1$ recovers a Gaussian measure.
For $1\leq q <3$,  the set of all $q$-Gaussian densities over the real line satisfies a certain regularity condition to define  information geometric structures such as an entropy and a relative entropy via escort expectations.
The ordinary expectation of a random variable  is  the integral of the random variable  with respect to its law.
Escort expectations admit us to replace  the law  to any other measures.
A choice of escort expectations on the set of all $q$-Gaussian densities determines an entropy and a relative entropy.
One of most important escort expectations on the set of all $q$-Gaussian densities is the $q$-escort expectation 
since this escort expectation determines the Tsallis entropy and the Tsallis  relative entropy.

The  phenomenon {\it gauge freedom of entropies} is that
different escort expectations determine the same entropy, but different relative entropies.
%
In this note, we first  introduce a refinement of the  $q$-logarithmic function.
Then we demonstrate the phenomenon on  an open set  of all $q$-Gaussian densities over the real line by using the refined $q$-logarithmic functions.
We  write down the corresponding  Riemannian metric.
\end{abstract}

\section{$q$-Logarithmic functions and their refinements}
\subsection{Definitions}
For $q\in\mathbb{R}$, we set $\chi_{q}:(0,\infty)\to (0,\infty)$ by
\[
\chi_q(s):=s^{q}.
\]
We define a strictly increasing function $\ln_q:(0,\infty) \to \mathbb{R}$ by
\[
\ln_{q}(t):
=\int_1^t \frac{1}{\chi_q(s)} ds
\]
and we denote by $\exp_{q}$  the inverse function of $\ln_{q}:(0,\infty) \to \ln_q(0,\infty)$.
The functions  $\ln_q$ and  $\exp_{q}$ are called the {\em $q$-logarithmic function} and  the {\em $q$-exponential function}, respectively.
We observe that 
\begin{align*}
\frac{d}{dt}\ln_{q}(t)&=\frac{1}{\chi_q(t)}=t^{-q} && \text{for\ } t\in(0,\infty),\\
\frac{d}{d\tau}\exp_{q}(\tau)&=\chi_q( \exp_q(\tau))=\exp_q(\tau)^{q} && \text{for\ } \tau\in \ln_q(0,\infty).
\end{align*}
It holds for $q\in \mathbb{R}$ that $\chi_q(1)=1$ and $\ln_q(1)=0$.
\begin{remark}
\begin{enumerate}
\item[(1)]
For $q=1$, we have that 
\begin{align*}
\ln_{1}(t)&=\log (t) \qquad \text{for\ }t\in (0,\infty),  \\
\ln_{1}(0,\infty)&=\mathbb{R},\\ 
\exp_{1}(\tau)&=\exp(\tau)\qquad \text{for\ }\tau \in \mathbb{R}.
\end{align*}
\item[(2)]
For $q\neq1$, we have that
\begin{align*}
\ln_q(t)&=\frac{t^{1-q}-1}{1-q} \qquad \text{for\ } t\in (0,\infty),  \\
\ln_{q}(0,\infty)&=
\begin{cases}
\displaystyle
\left(-\infty, \frac{1}{q-1}\right) & \text{if\ }q>1,\\[10pt]
\displaystyle
\left(-\frac{1}{1-q},\infty\right) & \text{if\ }q<1, 
\end{cases}\\
\exp_q(\tau)&=\{1+(1-q)\tau \}^{\frac1{1-q}} \qquad \text{for\ }\tau \in \ln_q(0,\infty).
\end{align*}
\end{enumerate}
\end{remark}

Taking account into the negativity of $\ln_q$ in $(0,1)$,  we  introduce  a refinement of  the $q$-logarithmic function and the $q$-exponential function.
For $q\in \mathbb{R}$ and ${a} \in \mathbb{R}\setminus\{0\}$,
define two functions $\chi_{q,{a}}: (0,1) \to (0,\infty)$ and $\ln_{q,{a}}: (0,1) \to \mathbb{R}$ respectively by
\[
\chi_{q,{a}}(s):=\chi_q(s) \cdot  (-\ln_q (s))^{1-{a}}, \qquad
\ln_{q,{a}}(t):=-\frac1{{a}}\big(-\ln_q(t) \big)^{{a}}.
\]
It turns out that 
\begin{align}\notag
\frac{d}{ds}\chi_{q,{a}}(s)&=
\chi_q'(s) (-\ln_q(s))^{1-{a}}
-
(1-{a}) (-\ln_q(s))^{-{a}} && \text{for\ $s\in (0,1)$,}\\   \label{deriv}
\frac{d}{dt}\ln_{q,{a}}(t)&=\frac{1}{\chi_{q,{a}}(t)}>0 && \text{for\ $t\in (0,1)$.}
\end{align}
Hence the function  $\ln_{q,{a}}: (0,1) \to \mathbb{R}$ is  strictly increasing.
We denote by $\exp_{q,{a}}$ the inverse function of $\ln_{q,{a}}:(0,1) \to \ln_{q,{a}}(0,1)$, which is give by 
\begin{equation}\label{ea}
\exp_{q,{a}}(\tau)=\exp_q\left(-\left(-{a} \tau\right)^\frac1{{a}} \right)\qquad
\text{for $\tau \in \ln_{q,{a}}(0,1)$.}
\end{equation}
The functions  $\ln_{q,{a}}$ and  $\exp_{q,{a}}$ are called 
the {\em ${a}$-refined $q$-logarithmic function} and  the  {\em ${a}$-refined $q$-exponential  function}, respectively.

On one hand, it holds for $q \geq1$ that 
\[
\ln_{q,{a}}(0,1)
=\begin{cases}
(-\infty,0) &\text{if\ } {a} >0,\\
(0,\infty) & \text{if\ }{a} <0.
\end{cases}
\]
 On the other hand, it holds for $q <1 $ that 
 \[
 \ln_{q,{a}}(0,1)=
\begin{cases}
\displaystyle
\left(-\frac1{{a}}(1-q)^{-{a}},0\right) &\text{if\ } {a} >0,\\[10pt]
\displaystyle
\left(-\frac1{{a}}(1-q)^{-{a}},\infty\right) &\text{if\ }{a} <0.
\end{cases}
 \]
\begin{remark}
\begin{enumerate}
\item[(1)]
The refinement of the ordinary logarithmic function, that is the case $q=1$, was introduced by Ishige, Salani and the second named author~\cite{IST-2019},
where they studied the preservation of concavity by the heat flow in Euclidean space.
\item[(2)]
For a positive function $\chi:(0,\infty) \to (0,\infty)$ and ${a} \in \mathbb{R}\setminus\{0\}$, 
the $\chi$-logarithmic function $\ln_\chi:(0,\infty) \to \mathbb{R}$ and 
its refinement $\ln_{\chi, {a}}:(0,1) \to \mathbb{R}$ are respectively defined in the same way as $\chi_q$.

\end{enumerate}
\end{remark}
\subsection{Properties}
In this section, we give a condition for $\ln_{q,{a}}$ to be concave and compute the higher order derivatives of $\exp_{q,{a}}$,
which will be used to define  information geometric structures.

For $q\in \mathbb{R}$ and ${a}\in \mathbb{R}\setminus\{0\}$,  define
\begin{align*}
&t_{q,{a}}:=
\begin{cases} 
0& \text{if either $q>0$ or $q=0$ with ${a}-1>0$},\\
1& \text{if $q\leq  0$ with ${a}-1 \leq  0$}, \\ 
\dfrac{1}{\exp_{q}\left(\frac{1-{a}}{q}\right)}& \text{otherwise},
\end{cases}\\
&
T_{q,{a}}:=
\begin{cases} 
0&
\text{if $q>1$ with\ }1-{a}\geq \dfrac{q}{q-1},\\
1&  \text{if $q \leq 0$},\\
\dfrac{1}{\exp_{q}\left(\max\left\{0, \frac{1-{a}}{q}\right\} \right)} & \text{otherwise},
\end{cases}
\end{align*}
and 
set $I_{q,{a}}:=(t_{q,{a}}, T_{q,{a}})$.
%
Note that 
$I_{q,{a}}$ is nonempty if and only if one of the following three conditions holds:
\begin{itemize}
\item[$\bullet$] $q>1$ with $1-{a}< \dfrac{q}{q-1}$;
\item[$\bullet$] $0<q\leq1$;
\item[$\bullet$] $q\leq 0$ with $a-1>0$.
\end{itemize}
\begin{proposition}\label{concave}
Fix $q\in \mathbb{R}$ and ${a}\in \mathbb{R}\setminus\{0\}$.
For an interval $I\subset (0,1)$,
the strict concavity of $\ln_{q,{a}}$ in $I$ is equivalent to the strict convexity of $\exp_{q,{a}}$ in $\ln_{q,{a}}(I)$.
Moreover,   if  $I_{q,{a}} \neq \emptyset$, 
then $\ln_{q,{a}}$ is strictly concave in~$I_{q,{a}}$.
\end{proposition}
\begin{proof}
Due to Equation \eqref{deriv},  $\ln_{q,{a}}$ is strictly increasing in $(0,1)$ and    so is  $\exp_{q,{a}}$ in $\ln_{q,{a}}(0,1)$.
Fix an interval $I\subset (0,1)$.
For
$t_i \in I, \tau_i\in \ln_{q,{a}}(I)$\  ($i=0,1$) with 
\[
\tau_i=\ln_{q,{a}}(t_i) \quad \text{or equivalently }\quad  t_i=\exp_{q,{a}}(\tau_i)
\]
and   $\lambda \in (0,1)$, 
it follows from the continuity of $\ln_{q,{a}}$ that 
\[
(1-\lambda) t_0 +\lambda t_1 \in I, \quad
(1-\lambda) \tau_0 +\lambda \tau_1 \in \ln_{q,{a}}(I).
\]
We observe from the monotonicity of $\ln_{q,{a}}$ and $\exp_{q,{a}}$ that 
\begin{align*}
&\ln_{q,{a}}\big( (1-\lambda) t_0 + \lambda t_1\big) > (1-\lambda) \ln_{q,{a}} (t_0) +  \lambda\ln_{q,{a}}(t_1)\\
\Leftrightarrow\ &
\ln_{q,{a}}\big( (1-\lambda) t_0 + \lambda t_1\big) >(1-\lambda) \tau_0 +\lambda \tau_1\\
\Leftrightarrow\ &
\exp_{q,{a}}\left( \ln_{q,{a}}\big( (1-\lambda) t_0 + \lambda t_1\big)\right)
 > \exp_{q,{a}}\left( (1-\lambda) \tau_0 +\lambda \tau_1 \right)\\
\Leftrightarrow\ &
(1-\lambda) t_0 +\lambda t_1 
 >\exp_{q,{a}}\left( (1-\lambda) \tau_0+  \lambda \tau_1 \right)\\
\Leftrightarrow\ &
(1-\lambda) \exp_{q,{a}}(\tau_0 )+\lambda \exp_{q,{a}}(\tau_1) 
 > \exp_{q,{a}}\left( (1-\lambda) \tau_0+  \lambda \tau_1 \right),
\end{align*}
where we used the fact that $\exp_{q,{a}}$ is the inverse function of  $\ln_{q,{a}}$.
This proves the first claim.

Assume $I_{q,{a}}\neq \emptyset$.
A direct calculation provides that 
\begin{align*}
\frac{d^2}{dt^2}\ln_{q,{a}} (t)
&=\frac{d}{dt}\frac{1}{\chi_{q,{a}} (t)} 
=-\frac{1 }{\chi_{q,{a}} (t)^2}\frac{d}{dt}\chi_{q,{a}}(t)\\
&=-\frac{(-\ln_q(t) )^{-{a}}}{\chi_{q,{a}} (t)^2}
\left\{\chi_q'(t) \left(-\ln_q(t)\right)  -(1-{a}) \right\}\\
&=
\frac{(-\ln_q(t) )^{-{a}}}{\chi_{q,{a}} (t)^2}
\left\{q t^{q-1}\ln_{q} (t)   +(1-{a}) \right\}.
\end{align*}
Notice that  $(-\ln_{q} (t) )^{-{a}}/\chi_{q,{a}} (t)^2$ is positive in $t\in I_{q,{a}}$.
In the case  $q=0$, 
the condition  $I_{0,{a}} \neq \emptyset$ leads to ${a}-1>0$,
consequently 
\[
\frac{d^2}{dt^2}\ln_{0,{a}} (t)
=\frac{(-\ln_0(t) )^{-{a}}}{\chi_{0,{a}} (t)^2} (1-{a})<0.
\]

Since the function  given by 
\[
t^{q-1} \ln_{q} (t) =-\ln_q \left(\frac1t\right)
=\begin{cases}
\log (t) &q=1,\\
\dfrac{1-t^{q-1}}{1-q} &q\neq 1
\end{cases}
\]
is strictly increasing in $ t\in (0,1)$, 
on one hand,  it holds for $q>0$ and $t\in I_{q,{a}}$ that 
\begin{align*}
\frac{d^2}{dt^2}\ln_{q,{a}} (t)
&
=
\frac{(-\ln_q(t) )^{-{a}}}{\chi_{q,{a}} (t)^2}
\left\{q t^{q-1}\ln_{q} (t)   +(1-{a}) \right\} \\
&
<
\frac{(-\ln_q(t) )^{-{a}}}{\chi_{q,{a}} (t)^2}
\left\{ -q \ln_q \left(\frac1{T_{q,{a}}}\right)   +(1-{a}) \right\} \\
&= 
\frac{(-\ln_q(t) )^{-{a}}}{\chi_{q,{a}} (t)^2}
\left\{-q \cdot \max\left\{0, \frac{1-{a}}{q}  \right\}+(1-{a}) \right\}\\
&=
 \frac{(-\ln_{q} (t) )^{-{a}}}{\chi_{q,{a}} (t)^2}
\left\{\min\left\{0, {a}-1  \right\} +(1-{a}) \right\}\\
&\leq 0.
\end{align*}
On the other hand, 
we see that
\begin{align*}
\frac{d^2}{dt^2}\ln_{q,{a}} (t)
&
=
\frac{(-\ln_q(t) )^{-{a}}}{\chi_{q,{a}} (t)^2}
\left\{q t^{q-1}\ln_{q} (t)   +(1-{a}) \right\} \\
&
<
\frac{(-\ln_q(t) )^{-{a}}}{\chi_{q,{a}} (t)^2}
\left\{ -q \ln_q \left(\frac1{t_{q,{a}}}\right)   +(1-{a}) \right\} \\
&= 
\frac{(-\ln_q(t) )^{-{a}}}{\chi_{q,{a}} (t)^2}
\left\{- q \cdot \frac{1-{a}}{q}  +(1-{a}) \right\}\\
&=0
\end{align*}
for $q<0$ and $t\in I_{q,a}$.
This completes the proof of the second claim.
\qquad$\qed$
\end{proof}

\begin{lemma}\label{seq}
For  $q\in \mathbb{R}$ and ${a}\in \mathbb{R}\setminus\{0\}$, 
there exists $\{b^n_j=b^n_j(q,{a})\}_{n\in \N, 0\leq j \leq n-1}$  such that 
\begin{align*}
\frac{d^n}{d\tau^n}\exp_{q,{a}}(\tau)
&=  \exp_{q,{a}}(\tau)^{(n-1)(q-1)+q} (-{a} \tau)^{\frac{n(1-{a})}{{a}}} \sum_{j=0}^{n-1} b_j^n(q,a) \cdot (-{a} \tau)^{-\frac{j}{a}}
\end{align*}
for  $\tau \in \ln_{q,{a}}(0,1)$.
Moreover, $\{b^n_j\}_{n\in \N, 0\leq j \leq n-1}$ satisfies 
\begin{align*}
b_0^1&=1,\\
b_j^{n+1}
&=\begin{cases}
\{n{a}(q-1)+1\} b_0^n &\text{if\ } j=0,\\
\{(n{a}+j)(q-1)+1\}b_j^n
-\{n(1-{a})-(j-1)\}b_{j-1}^n &\text{if\ }1\leq j \leq n-1,\\
(n{a}-1)b^n_{n-1} &\text{if\ } j=n.
\end{cases}
\end{align*}
\end{lemma}
\begin{proof}
We observe that 
\begin{align*}
\frac{d}{d\tau}\exp_{q,{a}}(\tau)
&=
\chi_{q,{a}}\left(\exp_{q,{a}}(\tau)\right)\\
&=\chi_q\left(\exp_{q,{a}}(\tau) \right) \cdot \left\{-\ln_q \left(\exp_{q,{a}}(\tau)\right)\right\}^{1-{a}}\\
&= \exp_{q,{a}}(\tau)^q \cdot (-{a} \tau)^{\frac{1-{a}}{a}},
\end{align*}
where we used Equation \eqref{ea}.
Thus the lemma holds for $n=1$.

If the lemma holds for $n$, then we compute that 
\begin{align*}
&\frac{d^{n+1}}{d\tau^{n+1}}\exp_{q,{a}}(\tau)\\
 =&
\frac{d}{d\tau}
\left(   \exp_{q,{a}}(\tau)^{(n-1)(q-1)+q} (-{a} \tau)^{\frac{n(1-{a})}{{a}}} \sum_{j=0}^{n-1} b_j^n \cdot (-{a} \tau)^{-\frac{j}{a}} \right)\\
=&
\left( \frac{d}{d\tau}   \exp_{q,{a}}(\tau)^{(n-1)(q-1)+q}  \right) \times (-{a} \tau)^{\frac{n(1-{a})}{{a}}} \sum_{j=0}^{n-1} b_j^n \cdot (-{a} \tau)^{-\frac{j}{a}} \\
&
+
 \exp_{q,{a}}(\tau)^{(n-1)(q-1)+q} \times
\frac{d}{d\tau} \left(  (-{a} \tau)^{\frac{n(1-{a})}{{a}}}  \sum_{j=0}^{n-1} b_j^n \cdot (-{a} \tau)^{-\frac{j}{a}} \right)\\
=&
\left\{ {(n-1)(q-1)+q}\right\}  \exp_{q,{a}}(\tau)^{(n-1)(q-1)+q-1}\cdot   \exp_{q,{a}}(\tau)^q  (-{a} \tau)^{\frac{1-{a}}{{a}}}  \\
&\qquad \times (-{a} \tau)^{\frac{n(1-{a})}{{a}}} \sum_{j=0}^{n-1} b_j^n \cdot (-{a} \tau)^{-\frac{j}{a}} \\
&
+
 \exp_{q,{a}}(\tau)^{(n-1)(q-1)+q} \times
\left\{ -{a} \sum_{j=0}^{n-1} \frac{ n(1-{a})-j}{a} b_j^n \cdot (-{a} \tau)^{\frac{ n(1-{a})-j}{a}-1} \right\}\\
=& \exp_{q,{a}}(\tau)^{n(q-1)+q} (-{a} \tau)^{\frac{(n+1)(1-{a})}{{a}}} \\
&  \times \ 
\left[
\{(n-1)(q-1)+q\} \sum_{j=0}^{n-1} b_j^n (-{a} \tau)^{-\frac{j}{a}} \right. \\
& \qquad\qquad 
    \left. -\exp_{q,{a}}(\tau)^{1-q}\sum_{j=0}^{n-1}\left\{n(1-{a})-j\right\}b_j^n (-{a} \tau)^{-\frac{j+1}{a}}\right].
\end{align*}
We deduce from 
$\exp_{q,{a}}(\tau)^{1-q} = 1-(1-q)(-{a} \tau)^{\frac1{{a}}}$ that
\begin{align*}
&\exp_{q,{a}}(\tau)^{1-q}
\sum_{j=0}^{n-1}
\left\{n(1-{a})-j\right\}b_j^n \cdot (-{a} \tau)^{-\frac{j+1}{a}}\\
&=
\sum_{j=0}^{n-1}
\left\{n(1-{a})-j\right\}b_j^n \cdot (-{a} \tau)^{-\frac{j+1}{a}}
-(1-q)\sum_{j=0}^{n-1}\left\{n(1-{a})-j\right\}b_j^n \cdot (-{a} \tau)^{-\frac{j}{a}}.
\end{align*}
This completes the proof of the lemma. 
\qquad$\qed$
\end{proof}
\begin{remark}\label{example}
For  $q\in \mathbb{R}$ and ${a}\in \mathbb{R}\setminus\{0\}$, we have that
\begin{align*}
&b_0^1=1, &&b_0^2={a}(q-1)+1, &&b_0^3=\{2{a}(q-1)+1\}\{{a}(q-1)+1\},\\
&\quad    &&b_1^2={a}-1, && b_1^3=({a}-1)\{(4{a}+1)(q-1)+3\},\\
&\quad&&\quad&&b_2^3=({a}-1)(2{a}-1).
\end{align*}
\end{remark}
\begin{corollary}\label{q=1}
For  ${a}\in \mathbb{R} \setminus\{0\}$ and  $n\in \N$, then $b_0^n(1,{a})=1$.
\end{corollary}
\begin{proof}
It follows from Lemma \ref{seq} that
\begin{align*}
b_0^{n+1}(1,{a}) 
=\{n{a}(1-1)+1\}b_0^n(1,{a}) 
= b_0^n(1,{a})=\cdots 
=b^1_0(1,{a})=1.
\qquad\qed
\end{align*}
\end{proof}
\begin{corollary}\label{a=1}
Let  $q\in \mathbb{R}$ and $n\in\mathbb{N}$.
For  $1\leq j<n$,  then $b_j^n(q,1)=0$.
\end{corollary}
\begin{proof}
This holds for $1=j<n=2$ by Remark \ref{example}.
For $n \geq 2$,
if $b_j^n(q,1)=0$ holds for  $1\leq j \leq n-1$,
then Lemma \ref{seq} implies
that $b_n^{n+1}(q,1)=(n{a}-1)b_{n-1}^{n}(q,1)=0$.
For $2 \leq j \leq n-1$, we have that
\begin{align*}
b_j^{n+1}(q,1)
&=
\{(n+j)(q-1)+1\}b_j^n(q,1)+(j-1)b_{j-1}^n(q,1)=0
\end{align*}
by the assumption $b^n_k(q,1)=0$  for $1 \leq k \leq n-1$.
For $j=1$, we have that  
\begin{align*}
b_1^{n+1}(q,1)
&=
\{(n+1)(q-1)+1\}b_1^n(q,1)+(1-1)b_{0}^n(q,1)=0.
\qquad
 \qed
\end{align*}
\end{proof} 

\section{Escort expectations}
The ordinary expectation of a random variable  is  the integral of the random variable  with respect to its law.
An introduction to escort expectations admits us to replace  the law  to any other measures.
The escort expectation with respect to a probability measure was first introduced by Naudts~\cite{Na-2004}.
\begin{definition}
For a measure $\nu$ on a measurable space $\Omega$, 
 the \emph{escort expectation} of  a function $f \in L^1(\nu)$ with respect to $\nu$ is defined by 
\[
\mathbb{E}_{\nu}[f]:=\int_{\Omega} f(\omega)  d\nu(\omega).
\]
\end{definition}

In this section, we fix a manifold $\mathcal{S}$ consisting of positive probability densities on a measure space $(\Omega, \nu)$.
We assume that  $\mathcal{S}$ is homeomorphic to an open set $\Xi$ in $\mathbb{R}^d$ and we denote each element in $\mathcal{S}$ by $p(\cdot;\xi)$ for $\xi \in \Xi$.
Namely,
\[
\mathcal{S}=\left\{  p(\cdot;\xi): \Omega \to  (0,\infty) \ \Big| \  \int_{\Omega}  p(\omega;\xi)d \nu(\omega) =1, \ \xi \in \Xi\right\}.
\]
We moreover require  that $\mathcal{S}$ satisfies  a certain regularity condition to define information geometric structures  via escort expectations.
For  the regularity condition, we refer to~\cite[Chapter 2]{AN-book}.

\begin{remark} \label{stat}
One  of   manifolds  consisting of probability densities on a measure space satisfying the regular condition 
is  a $q$-exponential family, which is a generalization to the space of $q$-Gaussian densities over $\mathbb{R}$ for $1\leq q<3$.
\end{remark}

Take $c\in(0,\infty]$ such that
\[
c > \sup\{ p(\omega)  \ |\ p\in \mathcal{S},\ \omega \in \Omega \}
\]
if  the above supremum is finite, otherwise $c:=\infty$.

\begin{definition}
Let $\ell: (0,c) \to \mathbb{R}$ be a differentiable function such that $\ell'>0$ in~$(0,c)$.
For $p\in \mathcal{S}$,
we  define a measure $\nu_{\ell;p}$ on $\Omega$ as the absolutely continuous measure with respect to $\nu$ with Radon--Nikodym derivative
\[
\frac{d\nu_{\ell;p}}{d\nu}(\omega)=\frac{1}{\ell'(p(\omega))}.
\]
\end{definition}
Note that $\ell$ is often assumed to be concave  such as the logarithmic function.
In the case $\ell=\log$, we have that 
\[
\frac{d\nu_{\ell;p}}{d\nu}=p.
\]

\begin{definition}\label{oridef}
Fix a  differentiable function $\ell: (0,c) \to \mathbb{R}$  such that $\ell'> 0$ in~$(0,c)$  and assume  that 
\begin{equation}\label{escortcondi}
\ell(r)=\ell  \circ r  \in L^1(\nu_{\ell;p} )\qquad \text{for\ } p, r \in \mathcal{S}.
\end{equation}
\begin{enumerate}
\renewcommand{\labelenumi}{$(\arabic{enumi})$}
\item
For $p, r \in \mathcal{S}$,
the {\em $\ell$-cross entropy} of $p$ with respect to $r$ is defined by
\[
d_{\ell}(p, r) 
:=- \mathbb{E}_{ \nu_{\ell;p}}[\ell  (r) ].
\]
\item
The \emph{$\ell$-entropy} of $p \in \mathcal{S}$ is defined by 
\[
\mathrm{Ent}_\ell(p):=d_{\ell}(p,p).
\]
\item
For $p, r \in \mathcal{S}$,  the {\em  $\ell$-relative entropy}  of $p$ with respect to $r$ is defined by 
\[
D^{(\ell)}(p,r):=-d_{\ell}(p,p)+d_{\ell}(p,r).
\]
\end{enumerate}
\end{definition}
\begin{remark}
In general,  the $\ell$-entropy does not satisfy 
nonextensive Shannon--Khinchin axioms~\cite{Su-2004}.
However, if  $\mathcal{S}$ is a manifold of all Gaussian densities over Euclidean space and $\ell=\log$, 
then the $\ell$-entropy coincides with the Boltzmann--Shannon entropy. 
\end{remark}

A choice of  differentiable functions $\ell: (0,c) \to \mathbb{R}$ such that $\ell'> 0$ in $(0,c)$ determines  an entropy and a relative entropy on $\mathcal{S}$. 
However, the converse is not  true.
This phenomenon is related to  {\it gauge freedom}, which was proposed by Zhang and Naudts \cite{ZN-2017} (see also \cite{NZ-2018}).

In the next section,
we demonstrate  {\it gauge freedom of entropies}  on an open set of $q$-Gaussian densities over $\mathbb{R}$ for $1\leq q <3$.
To be precise, we show that different escort expectations  determine the same entropy up to scalar multiple, but different relative entropies,
where  the entropy satisfies nonextensive Shannon--Khinchin axioms.

\section{Gauge freedom of Entropies}
\subsection{$q$-Gaussian measures}
To define $q$-Gaussian measures, 
we extend  $\exp_q$ to the whole of $\mathbb{R}$ by
\[
\mathrm{Rexp}_{q}(\tau):=\max\{0, 1+(1-q) \tau \}^{\frac{1}{1-q}}
\qquad
\text{for\ }\tau \in\mathbb{R},
\]
where by convention $0^c:=\infty$ for $c<0$.
We recall  the following improper integral. 
\begin{lemma}
For $q\in \mathbb{R}$ and $(\mu, \lambda )\in \mathbb{R} \times (0,\infty)$, the improper integral of the function
\[
 x \mapsto \mathrm{Rexp}_q(-\lambda (x-\mu)^2)
\]
on $\mathbb{R}$ converges if and only if $q<3$.
For $q<3$, 
\begin{align*}
\sqrt{3-q} \int_{\mathbb{R}} \mathrm{Rexp}_q(-x^2)dx = Z_q:=
\begin{cases}
\displaystyle \sqrt{\frac{{3-q}}{{q-1}}}B\left(\frac{3-q}{2(q-1)}  ,\frac12\right)&\text{if\ }q>1,\\
 \sqrt{2\pi} & \text{if\ }q=1,\\ 
 \displaystyle\sqrt{\frac{{3-q}}{{1-q}}}B\left(\frac{2-q}{1-q},\frac12\right)&\text{if\ }q<1,
 \end{cases}
\end{align*}
where $B(\cdot, \cdot)$ stands for the beta function.
\end{lemma}
\begin{proof}
By the change of variables, it is enough to show the case $(\mu, \lambda )=(0,1)$.
We omit the proof for the case $q=1$, which is well-known.

Assume $q\neq 1$.
There exist $c,C, R>0$ depending on $q$ such that 
\begin{align*}
c x^{\frac{2}{1-q}} \leq \mathrm{Rexp}_q(-x^2)=\left\{ 1-(1-q) x^2 \right\}^{\frac1{1-q}} <C x^{\frac{2}{1-q}}
\end{align*}
for $x>R$.
Since   the improper integral of the function
\[
 x \mapsto  x^{\frac{2}{1-q}}
\]
on $[1,\infty)$ converges if and only if $2/(1-q)<-1$, that is $q<3$, so does  the improper integral of the function
$ x \mapsto \mathrm{Rexp}_q(-x^2)$ on $\mathbb{R}$.

For $1<q<3$, we observe that 
\begin{align*}
\int_{\mathbb{R}} \mathrm{Rexp}_q(-x^2)dx 
&=
2\int_0^{\infty} \left\{1-(1-q)x^2\right\}^{\frac{1}{1-q}} dx\\
&=
\frac{1}{\sqrt{q-1}}\int_0^{\infty} (1+r)^{\frac{1}{1-q}}r^{-\frac12} dr\\
&=\frac{1}{\sqrt{q-1}}
B\left(\frac{3-q}{2(q-1)}  ,\frac12\right),
\end{align*}
where we used that 
\begin{align*}
B(s-t,t)=\int_{0}^\infty \frac{r^{s-1}}{(1+r)^{t}}dr\qquad\text{for \ }s>t>0.
\end{align*}

In the case $q<1$,  the support of the function $x \mapsto \mathrm{Rexp}_q(-x^2)$ on $\mathbb{R}$ is 
\[
\left[ -\frac{1}{\sqrt{1-q}}, \frac{1}{\sqrt{1-q}} \right]
\]
implying that 
\begin{align*}
\int_{\mathbb{R}} \mathrm{Rexp}_q(-x^2)dx 
&=
2\int_0^{\frac1{\sqrt{(1-q)}}} \left[1-(1-q)x^2\right]^{\frac{1}{1-q}} dx\\
&=
\frac{1}{\sqrt{1-q}}\int_0^{1} \left[1-r\right]^{\frac{1}{1-q}}r^{-\frac12} dx\\
&=\frac{1}{\sqrt{1-q}}B\left(\frac{2-q}{1-q},\frac12\right).
\qquad  \qed
\end{align*}
\end{proof}

\begin{definition}
For $q<3$ and  $\xi=(\mu,\sigma) \in \mathbb{R}\times(0,\infty)$, the \emph{$q$-Gaussian measure} with location parameter $\mu$
and scale parameter $\sigma$ on $\mathbb{R}$ is an absolutely continuous probability measure with respect to the one-dimensional Lebesgue measure with Radon--Nikodym derivative
\[
p_{q}(x;\xi)=
p_{q}(x;\mu, \sigma):=
\frac{1}{Z_q \sigma}\mathrm{Rexp}_q\left(-\frac1{3-q}\left(\frac{x-\mu}{\sigma}\right)^2 \right).
\]
We call  $p_{q}(x;\xi)=p_{q}(x;\mu, \sigma)$ the {\em $q$-Gaussian density} with {location parameter} $\mu$ and scale parameter $\sigma$.
\end{definition}

%
A $q$-Gaussian density   corresponds to a \emph{normal $($Gaussian$)$distribution} for $q=1$,  and a \emph{Student $t$-distribution} for $1<q<3$.
In the both cases, the support of each $q$-Gaussian measure is the whole of $\mathbb{R}$ and 
\[
p_{q}(x;\xi)=
p_{q}(x;\mu, \sigma)=
\frac{1}{Z_q \sigma}\exp_q\left(-\frac1{3-q}\left(\frac{x-\mu}{\sigma}\right)^2 \right).
\]
The set of all $q$-Gaussian densities   satisfies the regularity condition to define  information geometric structures.
For example, see~\cite{MW-2015}.

\subsection{Sufficient conditions for \eqref{escortcondi}}
In order to give a rigorous treatment of an escort expectation  associated to the ${a}$-refined $q$-logarithmic function, we only deal with the case  $1\leq q<3$.
Set
\[
\Sigma_{q}:=\left\{\sigma>0\  \Big|\  \frac{1}{Z_q\sigma}  < 1\right\},\qquad
\mathcal{S}_q:=\{p_q(\cdot;\xi)\ \ |\  \xi \in \mathbb{R}\times \Sigma_q\}.
\]
%
It holds for $\sigma \in \Sigma_{q}, p\in \mathcal{S}_q$ and $x\in \mathbb{R}$ that 
\[
\ln_q\left( \frac1{Z_q \sigma} \right)< \ln_q(1)=0, \qquad
\ln_q\left(p(x)\right)\in (-\infty, 0).
\]
\begin{definition}
For $1\leq q<3$ and $\xi \in\mathbb{R}\times \Sigma_{q}$, 
define $\ell_q(\cdot;\xi):\mathbb{R}\to (-\infty,0)$  by 
\begin{align*}
\ell_q(x;\xi):=\ln_{q}\left( p_q(x;\xi)\right),
\end{align*}
which is called the {\em $q$-likelihood function} of $p_q(\cdot;\xi)$.
\newpage

For $1\leq q<3,  {a}\in \mathbb{R}\setminus\{0\}$ and $\xi \in\mathbb{R}\times \Sigma_{q}$, we  define a measure $\nu_{q,{a};\xi}$ on $\mathbb{R}$
as the absolutely continuous  measure with respect to the one-dimensional Lebesgue measure with Radon--Nikodym derivative
\[
\frac{d\nu_{q,{a};\xi}}{dx}(x)=
\frac1{\ln_{q,{a}}'\left(p_q(x;\xi) \right)}.
\]
\end{definition}
Since  the inverse function of $\ln_{q,a}$ is $\exp_{q,a}$,  Lemma \ref{seq} in the case $n=1$ leads to
\begin{align*}
\frac{d\nu_{q,{a};\xi}}{dx}(x)=\exp_{q,{a}}'(\ln_{q,{a}}(p_q(x;\xi)) )&=\left(- \ell_q(x;\xi )\right)^{1-{a}} \chi_q (p_q( x;\xi )).
\end{align*}
A direct computation leads to the relation that 
\begin{equation}
\label{eq:sub}
\begin{split}
\ell_q(x;\xi)& =\ln_q\left( \frac1{Z_q \sigma} \right)-\frac{1}{(Z_q\sigma)^{1-q}(3-q)}\left(\frac{x-\mu}{\sigma}\right)^2,\\
\ln_{q,{a}}( p_q(x;\xi))&=-\frac1{{a}}\left(-\ell_q(x;\xi)\right)^{{a}}.
\end{split}
\end{equation}
\begin{lemma}\label{l2}
Let    $1\leq q <3, {a} \in \mathbb{R}\setminus\{0\}$ and $\xi\in \mathbb{R} \times \Sigma_q$.
Then  for $\lambda,\gamma\in\mathbb{R}$ with $\lambda>0$, $(\lambda+ x^2)^{\gamma} \in L^1(\nu_{q,{a};\xi})$ if and only if
\[
\text{either\ } q=1\quad \text{or }\quad q>1\text{\ with\ } \gamma < \frac12 +\frac{1}{q-1}+{a}-1.
\]
\end{lemma}
\begin{proof}
Since the decay rate of $\nu_{1,{a};\xi}$ is $o(\exp(-x^{\varepsilon}))$ as $x\to \infty$  for $\varepsilon<2$,  the lemma  holds for $q=1$.

Assume $q> 1$.
By the change of variables, it is enough to show the case $\xi=(0,2/Z_q)$.
Here we have that $Z_q \sigma=2$.
There exist $c,C,R>0$ depending on~$q$  such that 
\begin{align*}
& c x^{2(1-{a})+\frac{2q}{1-q}+2\gamma}\\
 &<
\left(- \ell_q(x;\xi)\right)^{1-{a}}\cdot \chi_q(p_q(x;\xi)) \cdot  (\lambda+x^2)^{\gamma} \\
&
=
\left\{-\ln_q\left(\frac12 \right)+\frac{1}{ 2^{1-q}(3-q)}\frac{Z_q^2x^2}{4}\right\}^{1-{a}} \cdot \frac1{2^{q}} \left(   1+\frac{q-1}{3-q} \frac{Z_q^2x^2}{4} \right)^{\frac{q}{1-q}} \cdot  (\lambda+x^2)^{\gamma}  \\
 &< 
 C  x^{2(1-{a})+\frac{2q}{1-q}+2\gamma}
 \end{align*}
for $x> R$.
This means  that $(c+x^2)^\gamma \in L^1(\nu_{q,{a};\xi}) $ 
if and only if 
\begin{align*}
2(1-{a})+\frac{2q}{1-q}+2\gamma<-1
\ \Leftrightarrow\ 
\gamma<\frac12+\frac{1}{q-1}+{a}-1.
\qquad
\qed
\end{align*}

\end{proof}

Lemma \ref{l2} in the case $\gamma=0$ provides the condition for $(q,{a})$ such that  $\nu_{q,{a};\xi}$ has a finite mass.
\begin{corollary}
Let  $1\leq q <3, {a} \in \mathbb{R}\setminus\{0\}$ and $\xi\in \mathbb{R} \times \Sigma_q$.
Then $1 \in L^1(\nu_{q,{a};\xi})$ if and only if
\[
\text{either\ } q=1\quad \text{or }\quad q> 1\text{\ with\ } 
\frac12-\frac{1}{q-1}<{a}.
\]
\end{corollary}
\newpage
Note that 
\[
\frac12-\frac{1}{q-1}<0
\qquad 
\text{for\ }1<q<3.
\]

\begin{corollary}
Let $1\leq q <3$ and ${a} \in \mathbb{R}\setminus\{0\}$.
Then  $\ln_{q,{a}}(r)  \in L^1( \nu_{q,{a};\xi})$ for  $\xi \in \mathbb{R} \times \Sigma_q$ and  $r\in \mathcal{S}_q$.
 \end{corollary}
\begin{proof}
The corollary trivially holds for $q=1$.

Assume $q>1$.
We observe from \eqref{eq:sub} that 
\begin{align*}
\ln_{q,{a}}( p(x;\mu,\sigma))
=-\frac1a\left\{-\ln_q\left( \frac1{Z_q \sigma} \right)+\frac{1}{(Z_q\sigma)^{1-q}(3-q)}\left(\frac{x-\mu}{\sigma}\right)^2 \right\}^{a}
\end{align*}
for $(\mu,\sigma) \in\mathbb{R}\times \Sigma_q$.
This with Lemma \ref{l2} yields that 
\begin{align} \label{condition}
\ln_{q,{a}}(r) \in L^1( \nu_{q,{a};\xi}) 
\ \Leftrightarrow \ 
a<\frac12+\frac{1}{q-1}+a-1,
\end{align}
which holds  for  $q<3$.
\qquad
$\qed$
\end{proof}

Following Definition \ref{oridef}, we define an entropy  and  a relative entropy on $\mathcal{S}_q$.
Recall  the escort expectation of  a function $f \in L^1(\nu_{q,{a};\xi})$ with respect to $\nu_{q,{a};\xi}$ is defined by 
\[
\mathbb{E}_{\nu_{q,{a};\xi}}[f]=\int_{\mathbb{R}} f(x)  d\nu_{q,{a};\xi}(x)=\int_{\mathbb{R}} f(x) \exp_{q,{a}}'\big(\ln_{q,{a}}\left(p_q(x;\xi) \right)\big) dx.
\]

\begin{definition}\label{def}
Let $1\leq q <3$ and ${a} \in \mathbb{R}\setminus\{0\}$.
Take $\xi\in \mathbb{R} \times \Sigma_q$ and set $p=p_q(\cdot;\xi)\in \mathcal{S}_q$.
\begin{enumerate}
\renewcommand{\labelenumi}{$(\arabic{enumi})$}
\item
The {\em $(q,{a})$-cross entropy} of   $p$ with respect to $r\in \mathcal{S}_q$ is defined by 
\[
d_{q,{a}}(p, r)
:=-\mathbb{E}_{ \nu_{q,{a};\xi}}[\ln_{q,{a}}(r) ] .
\]
\item
The \emph{$(q,{a})$-entropy} of $p$ is defined by 
\[
\mathrm{Ent}_{q,{a}}(p):=d_{q,{a}}(p,p).
\]
\item
The {\em  $(q,{a})$-relative entropy} of  $p$ with respect to $r\in \mathcal{S}_q$  is defined by 
\[
D^{(q,{a})}(p,r)
:=-d_{q,{a}}(p,p)+d_{q,{a}}(p,r).
\]
\end{enumerate}
\end{definition}

\begin{remark}
The domain of  the $(q,1)$-entropy can be extended to the whole of  $q$-Gaussian densities.
The $(q,1)$-entropy  coincides with the \emph{Boltzmann--Shannon entropy} if $q=1$,  and  the \emph{Tsallis entropy} otherwise.
\end{remark}

\begin{theorem}[gauge freedom of entropies]\label{thm}
Let   $1\leq q <3$ and ${a} \in \mathbb{R}\setminus\{0\}$.
Then 
\begin{align*}
\mathrm{Ent}_{q,1}=a\mathrm{Ent}_{q,{a}}, \qquad
D^{(q,1)}\neq  \lambda D^{(q,{a})}
\qquad
\text{for\ }a\neq 1 \ \text{\ and\ }\ \lambda \in \mathbb{R}.
\end{align*}
\end{theorem}
\begin{proof}
By the definition, we have that 
\begin{align*}
d_{q,{a}}(p_q(\cdot;\xi_0), p_q(\cdot;\xi))
&=\frac1a\int_{\mathbb{R}} \left(-\ell_q(x;\xi)\right)^{a} \nu_{q,{a};\xi_0}(x)\\
&=\frac1a\int_{\mathbb{R}} \left(-\ell_q(x;\xi)\right)^{a}\left(-\ell_q(x;\xi_0)\right)^{1-{a}} \chi_q(p_q(x;\xi_0))dx
\end{align*}
for $\xi_0, \xi\in\mathbb{R}\times \Sigma_q$, which implies that 
\begin{align*}
\mathrm{Ent}_{q,1}(p)=a\mathrm{Ent}_{q,{a}}(p)
&=-\int_{\mathbb{R}} \ln_{q} (p(x)) p(x)^q dx \\
&=\begin{cases}
\displaystyle
-\int_{\mathbb{R}} \frac{p(x)-p(x)^q}{1-q}  dx & q> 1,\\[10pt]
\displaystyle
-\int_{\mathbb{R}} p(x) \log (p(x)) dx& q=1
\end{cases}
\end{align*}
for $p\in \mathcal{S}_q$.

Recall that $\Sigma_q=\{\sigma>0 \ |\ \sigma >1/Z_q  \}$.
Since we observe that 
\[
\lim_{\sigma \to \infty} \frac{\left(-\ell_q(x;0,\sigma)\right)^{a}}{ -\ln_q\left(\frac1{Z_q \sigma}\right)}=
\begin{cases}
\infty & \text{if \ }{a}>1,\\
1& \text{if \ }{a}=1,\\
0& \text{if \ }{a}<1, {a} \neq 0, 
\end{cases}
\]
we apply the dominated convergence theorem ${a}\leq 1$ and 
the  monotone convergence theorem for ${a} >1$ to have 
\begin{align*}
&\frac{\lambda D^{(q,a)} (p, p_q(\cdot;0;\sigma))-D^{(q,1)}(p, p_q(\cdot;0;\sigma))}{ -\ln_q\left(\frac1{Z_q \sigma}\right)}\\
&=
- \frac{\lambda d_{q,a} (p, p)-d_{q,1}(p,p)}{ -\ln_q\left(\frac1{Z_q \sigma}\right)}+\frac{\lambda d_{q,a} (p, p_q(\cdot;0;\sigma))-d_{(q,1)}(p, p_q(\cdot;0;\sigma))}{ -\ln_q\left(\frac1{Z_q \sigma}\right)}\\
&\xrightarrow {\sigma \to \infty}
\begin{cases}
\lambda \cdot \infty-c & \text{if \ }{a}>1,\\
(\lambda-1)c& \text{if \ }{a}=1,\\
-c& \text{if \ }{a}<1, {a} \neq 0
\end{cases}
\end{align*}
for $p\in \mathcal{S}_q$ and $\lambda \in \mathbb{R}$,
where  we put $ 0\cdot \infty:=0$ and
\[
c:= \int_{\mathbb{R}}  \chi_q( p(x) ) dx.
 \]
This constant $c$ is obviously positive, and $c$ is finite due to Lemma \ref{l3} in the next section.
This ensures that $D^{({q,{a}})} \neq \lambda D^{(q,1)}$ for $a\neq 1$ and  $\lambda \in \mathbb{R}$.
\qquad
$\qed$
\end{proof}
The proof of Theorem \ref{thm}  immediately gives the following corollary.
\begin{corollary}
Let   $1\leq q <3$ and ${a} \in \mathbb{R}\setminus\{0\}$.
Then 
\begin{align*}
d_{q,1}\neq \lambda d_{q,{a}}
\qquad
\text{for\ }a\neq 1\  \text{\ and\ }\ \lambda \in \mathbb{R}.
\end{align*}
\end{corollary}

\section{Refined  Riemannian metrics}
Throughout of this section, we fix  $ 1\leq q <3$ and ${a}\in \mathbb{R}\setminus\{0\}$ such that $I_{q,{a}} \neq \emptyset$, namely 
\[
\text{either\ } q=1\quad 
\text{or }\quad 
q>1 \text{\ with\ } 1-{a}< \dfrac{q}{q-1}.
\]
In this case, $t_{q,{a}}=0$.
Set 
\[
\Sigma_{q,{a}}:=\left\{\sigma \in \Sigma_q \  \big|\  \frac{1}{Z_q\sigma}<T_{q,{a}}\right\}, \qquad
\mathcal{S}_{q,{a}}:=\left\{p_q(\cdot; \xi)  \in \mathcal{S}_q \ |\ \xi \in \mathbb{R}\times \Sigma_{q,{a}}  \right\}.
\]
The manifold $\mathcal{S}_{q,{a}}$ admits information geometric structures.
\subsection{Derivatives of $(q,a)$-relative entropy}
The $(q,{a})$-relative entropy is  nondegenerate on  $\mathcal{S}_{q,{a}} \times \mathcal{S}_{q,{a}} $.
\begin{lemma}\label{cor}
For $p,r\in \mathcal{S}_{q,{a}}$,  $ D^{(q,{a})}(p,r) >0$.
\end{lemma}
\begin{proof}
Proposition \ref{concave} yields that $\exp_{q,{a}}''(\ln_{q,{a}}(p(x)))>0$  in $x\in \mathbb{R}$ for $p\in \mathcal{S}_{q,{a}}$.
The strict convexity of $\exp_{q,{a}}$ leads to the inequality that 
\begin{align*}
r(x)&=
\exp_{q,{a}}(\ln _{q,{a}}(r(x))  ) \\
&> \exp_{q,{a}}(\ln _{q,{a}}(p(x)))+\left\{\ln _{q,{a}}(r(x))-\ln _{q,{a}}(p(x))\right\} \exp_{q,{a}}'(\ln _{q,{a}}(p(x)))\\
&=p(x)+\ln _{q,{a}}(r(x)) \exp_{q,{a}}'(\ln _{q,{a}}(p(x)))
-\ln _{q,{a}}(p(x))  \exp_{q,{a}}'(\ln _{q,{a}}(p(x)))
\end{align*}
for $x\in \mathbb{R}$ and $p,r\in \mathcal{S}_{q,{a}}$.
Integrating this inequality on $\mathbb{R}$ gives
\[
1>1-d_{q,{a}}(p,r)+d_{q,{a}}(p,p)=1-D^{(q,{a})}(p,r).
\qquad
\qed
\]
\end{proof}

Let us define a function  $\rho^{(q,{a})}$ on  $(x,\xi_1,\xi_2)\in \mathbb{R} \times(\mathbb{R}\times \Sigma_{q,{a}})^2$ by 
\begin{align*}
\rho^{(q,{a})}(x; \xi_1,\xi_2)
:= \left\{\ln_{q,{a}}(p_q(x;  \xi_1))-\ln_{q,{a}}(p_q(x;  \xi_2))\right\} \exp'_{q, {a}} \left(\ln_{q,{a}}(p_q(x; \xi_1))\right),
\end{align*}
which is  the integrand of $D^{(q,{a})} (p_q(\cdot;\xi_1),p_q(\cdot;\xi_2))$.

Given $\xi_i=(\mu_i, \sigma_i) \in \mathbb{R}\times \Sigma_{q,{a}}$, it turns out that 
\begin{align*}
& \frac{\partial}{\partial s_1} \frac{\partial}{\partial s_2}\rho^{(q,{a})}(x;\xi_1, \xi_2) \Big|_{(\xi,\xi)} \\
&=-\frac{\partial}{\partial s_2}\ln_{q,{a}}\left( p_q \left(x; \xi_2 \right)\right) \cdot \frac{\partial}{\partial s_1}  \exp'_{q, {a}} \left(\ln_{q,{a}}\left(p_q(x;\xi_1\right)\right) \Big|_{(\xi,\xi)}\\
&=-\frac{\partial}{\partial s_2}\ln_{q,{a}}\left( p_q \left(x; \xi_2 \right)\right) \cdot \frac{\partial}{\partial s_1} \ln_{q,{a}}\left(p_q(x;\xi_1\right)
\cdot  \exp''_{q, {a}} \left(\ln_{q,{a}}\left(p_q(x;\xi_1\right)\right) \Big|_{(\xi,\xi)}\\
&=
-\frac{\partial}{\partial s_2}\left\{-\frac{1}{{a}}\left(-\ell_q(x;\xi_2)\right)^{{a}} \right\}\cdot
\frac{\partial}{\partial s_1}\left\{-\frac{1}{{a}}\left(-\ell_q(x;\xi_1)\right)^{{a}} \right\} \bigg|_{(\xi,\xi)} \\
& \qquad\quad \times 
  p_q(x;\xi)^{(2-1)(q-1)+q} \left(-\ell_q(x;\xi)\right)^{2(1-{a})} \sum_{j=0}^1 b_j^2 \left( -\ell_q(x;\xi)\right)^{-j} \\
&=-\sum_{j=0}^1 b_j^2
\left( \frac{\partial}{\partial s_2} \ell_q(x;\xi_2) \cdot
\frac{\partial}{\partial s_1}\ell_q(x;\xi_1) \bigg|_{(\xi,\xi)}\cdot  \left(-\ell_q(x,\xi)\right)^{-j}  p_q(x;\xi)^{2q-1}   \right)
\end{align*}
for  $s_i\in \{ \mu_i, \sigma_i\}$, 
where we used Lemma \ref{seq} in the case $n=2$.

Let us  generalize Lemma \ref{l2}.
\begin{lemma}\label{l3}
Fix  $n\in \mathbb{N}$ and $\gamma\geq 0$. 
Then $\exp_q(-x^2)^{(n-1)(q-1)+q} \cdot x^{2\gamma}\in L^1(dx)$
if and only if 
\[
\text{either\ } q=1\quad \text{or }\quad q>1\text{\ with\ } \gamma<\frac12+\frac{1}{q-1}+n-1.
\]
\end{lemma}
\begin{proof}
The lemma trivially holds for $q=1$.
Assume $q>1$.
There exist $c,C,R>0$ depending on $q$ such that 
\begin{align*}
&c x^{2\frac{(n-1)(q-1)+q}{1-q}+2\gamma}\\
&<
\exp_q(-x^2)^{(n-1)(q-1)+q}\cdot  x^{2\gamma}
=
\left\{1-(1-q)x^2\right\}^{\frac{(n-1)(q-1)+q}{1-q}}\cdot  x^{2\gamma} \\
 &< 
C x^{2\frac{(n-1)(q-1)+q}{1-q}+2\gamma}
\end{align*}
for $x>R$.
This yields that $\exp_q(-x^2)^{(n-1)(q-1)+q} x^{2\gamma} \in L^1(dx) $ 
if and only if 
\begin{align*}
2\frac{(n-1)(q-1)+q}{1-q}+2\gamma<-1
\ \Leftrightarrow\  
\gamma<\frac12+\frac{1}{q-1}+n-1.
\qquad
\qed
\end{align*}
\end{proof}
\begin{corollary}\label{c3}
For $n\in \mathbb{N}, 0\leq \gamma\leq n, j\in\mathbb{Z}_{\geq 0}$ and $\xi\in\mathbb{R} \times \Sigma_{q,{a}}$, 
then
\[ 
  p_q(x;\xi)^{(n-1)(q-1)+q}  \cdot x^{2\gamma} \cdot  \left(-\ell_q(x;\xi)\right)^{-j} \in L^1(dx).
\]
\end{corollary}
\begin{proof}
Since we have that
\[
n<\frac12+\frac{1}{q-1}+n-1 \qquad \text{for\ }1<q<3,
\]
we apply Lemme~\ref{l3} together with the change of variables to have that
\[ 
  p_q(x;\xi)^{(n-1)(q-1)+q}   \cdot  x^{2\gamma}\in L^1({dx}) \qquad \text{for\ } 0\leq\gamma \leq n.
\]
Moreover, the fact that 
\[
-\ell_q(x;\xi)\geq -\ln_q \left(\frac{1}{Z_q \sigma}\right)>0
\]
completes  the proof of the corollary.
\qquad
$\qed$
\end{proof}
Combining the computation that 
\begin{align}
\begin{split}\label{der}
\frac{\partial}{\partial \mu} \ell_q(x;\mu,\sigma)
&=\frac{2}{(3-q)}\cdot\frac1{(Z_q\sigma)^{1-q}\sigma}\frac{x-\mu}{\sigma}, \\
\frac{\partial}{\partial \sigma} \ell_q(x;\mu,\sigma)
&=-\frac{1}{(Z_q\sigma)^{1-q}\sigma}\left\{1-\left(\frac{x-\mu}{\sigma}\right)^2\right\}
\end{split}
\end{align}
with Corollary \ref{c3} in the case $n=2$, we conclude that 
\[
x\mapsto  \frac{\partial}{\partial s_1} \frac{\partial}{\partial s_2}\rho^{(q,{a})}(x;\xi_1, \xi_2)\Big|_{(\xi,\xi)}
\]
is integrable on $\mathbb{R}$ for $\xi\in \mathbb{R}\times \Sigma_{q,{a}}$.
Since the function $x\mapsto  \rho^{(q,{a})}(x;\xi_1, \xi_2)$ is integrable on $\mathbb{R}$ for  $(\xi_1, \xi_2)\in (\mathbb{R}\times \Sigma_{q,{a}})^2$,
the dominated convergence theorem implies that 
\begin{align*}
& \frac{\partial}{\partial s_1} \frac{\partial}{\partial s_2} D^{(q,{a})} (p_q(\cdot;\xi_1),p_q(\cdot;\xi_2)\Big|_{(\xi,\xi)}\\
&=-\int_{\mathbb{R}}\frac{\partial}{\partial s_2}\ln_{q,{a}}\left( p_q \left(x; \xi_2 \right)\right) \cdot \frac{\partial}{\partial s_1} \ln_{q,{a}}\left(p_q(x;\xi_1\right)
\cdot  \exp''_{q, {a}} \left(\ln_{q,{a}}\left(p_q(x;\xi_1\right)\right) \Big|_{(\xi,\xi)}dx\\
&=-\sum_{j=0}^1 b_j^2 \int_{\mathbb{R}}
\left( \frac{\partial}{\partial s_2} \ell_q(x;\xi_2) \cdot
\frac{\partial}{\partial s_1}\ell_q(x;\xi_1) \right)\bigg|_{(\xi,\xi)}\cdot  \left(-\ell_q(x,\xi)\right)^{-j}  p_q(x;\xi)^{2q-1}  dx
\end{align*}
for $s_i\in \{\mu_i, \sigma_i\}$.
This quantity evaluated at the diagonal set  $\{(\xi_1,\xi_2) \ |\ \xi_1=\xi_2 \}$ provides a Riemannian metric on $\mathcal{S}_{q,{a}}$.
\begin{definition}
For $s,t\in\{\mu, \sigma\}$, 
define a function $g^{(q,{a})}_{st}:\mathbb{R} \times \Sigma_{q,{a}}\to\mathbb{R} $ by
\begin{align*}
g^{(q,{a})}_{st}(\xi)
:=\int_{\mathbb{R}} 
\frac{\partial}{\partial s}\ln_{q,{a}} \left(p_q(x;\xi)\right)\cdot
\frac{\partial}{\partial t}\ln_{q,{a}} \left(p_q(x;\xi)\right)\cdot
\exp''_{q,{a}}\left(\ln_{q,{a}} \left(p_q(x;\xi)\right)\right) dx.
\end{align*}
\end{definition}

\begin{theorem}
For $\xi\in \mathbb{R} \times \Sigma_{q,{a}}$ and $s,t\in\{\mu, \sigma\}$, 
\[
g^{(q,{a})}\left(\frac{\partial}{\partial s},\frac{\partial}{\partial t} \right) (p_{q}(\cdot;\xi))
:=
g^{(q,{a})}_{st}(\xi)
\]
 determines a Riemannian metric on $\mathcal{S}_{q,{a}}$.
\end{theorem}
\begin{proof}
It is enough to show that
\[
g^{(q,{a})}_{\mu\mu}, g^{(q,{a})}_{\sigma\sigma}>0\quad \text{\ and\ }\quad
g^{(q,{a})}_{\mu\sigma}=0
\quad\text{\ on\ }  \mathbb{R} \times \Sigma_{q,{a}}.
\]
The positivities of $g^{(q,{a})}_{\mu\mu}, g^{(q,{a})}_{\sigma\sigma}$ follows from that of 
\[
\frac{\partial}{\partial s}\ln_{q,{a}} \left(p_q(x;\xi)\right)\cdot
\frac{\partial}{\partial s}\ln_{q,{a}} \left(p_q(x;\xi)\right)\cdot
\exp''_{q,{a}}\left(\ln_{q,{a}} \left(p_q(x;\xi)\right)\right)
\qquad
\text{for\ }s\in \{\mu,\sigma\}.
\]
We derive $g^{(q,{a})}_{\mu\sigma}=0$ from the fact that 
\[
\frac{\partial}{\partial \mu}\ln_{q,{a}} \left(p_q(x;\xi)\right)\cdot
\frac{\partial}{\partial \sigma}\ln_{q,{a}} \left(p_q(x;\xi)\right)\cdot
\exp''_{q,{a}}\left(\ln_{q,{a}} \left(p_q(x;\xi)\right)\right)
\]
is an odd function in $x\in \mathbb{R}$ with respect to $x=\mu$ according to \eqref{der}.
\qquad
$\qed$
\end{proof}

\begin{remark}
The  Riemannian metric $g^{(q,1)}$  coincides with the \emph{Fisher metric}  up to scalar multiple.
The third order derivatives of  $(q,1)$-relative entropy on the set of all $q$-Gaussian densities induce a pair of affine connections.
The  cubic tensor  which expresses the difference between the two affine connections is called the  {\em Amari--\v{C}encov tensor}.
In a similar way, a cubic tensor $C^{(q,{a})}$  is  defined by 
\begin{align*}
&C^{(q,{a})}\left(\frac{\partial}{\partial s},\frac{\partial}{\partial t},\frac{\partial}{\partial u}\right)(p_q(\cdot;\xi))
\\
&:=
\int_{\mathbb{R}} 
\frac{\partial}{\partial s}\ln_{q,{a}} \left(p_q(x;\xi)\right)\cdot
\frac{\partial}{\partial t}\ln_{q,{a}} \left(p_q(x;\xi)\right)\cdot
\frac{\partial}{\partial u}\ln_{q,{a}} \left(p_q(x;\xi)\right)\\
& \qquad\qquad\qquad\qquad \qquad\qquad\qquad\quad  \times 
\exp'''_{q,{a}}\left(\ln_{q,{a}} \left(p_q(x;\xi)\right)\right) dx\\
&=\int_{\mathbb{R}} 
\frac{\partial}{\partial s}\left\{-\frac{1}{{a}}\left(-\ell_q(x;\xi)\right)^{{a}} \right\}\cdot
\frac{\partial}{\partial t}\left\{-\frac{1}{{a}}\left(-\ell_q(x;\xi)\right)^{{a}} \right\}\cdot
\frac{\partial}{\partial u}\left\{-\frac{1}{{a}}\left(-\ell_q(x;\xi)\right)^{{a}} \right\}  \\
& \qquad\qquad\qquad\qquad \times 
p_q(x;\xi)^{(3-1)(q-1)+q} \left(-\ell_q(x;\xi)\right)^{3(1-{a})}\sum_{j=0}^{2}b_j^2\left(-\ell_q(x;\xi)\right)^{-j}\\
&= \sum_{j=0}^2 b_j^3 \int_{\mathbb{R}} 
\frac{\partial}{\partial s}\ell_q(x;\xi) \cdot
\frac{\partial}{\partial t}\ell_q(x;\xi) \cdot
\frac{\partial}{\partial u}\ell_q(x;\xi) \cdot \left(- \ell_q(x;\xi)\right)^{-j}
p_q(x;\xi)^{3q-2}  dx
\end{align*}
for $s,t,u \in \{\mu,\sigma\}$ and $\xi \in \mathbb{R} \times \Sigma_{q,{a}}$.
The above  improper integral converges due to Corollary \ref{c3} in the case $n=3$.

The Fisher metric (resp. the Amari--\v{C}encov tensor) is a unique invariant  quadric (resp. cubic) tensor  under Markov embeddings up to scalar multiple (see~\cite[Chapter 5]{AJLS-book}).
\end{remark}
\subsection{Expression of the refined Riemann metrics}
We compute the exact value of 
\begin{equation}
\label{eq:4.1}
\begin{split}
g^{(q,{a})}_{\mu \mu}(\xi)
&=\frac{4}{(3-q)^2}\sum_{j=0}^1   \frac{b_j^2}{(Z_q\sigma)^{2(1-q)}\sigma^2}
\int_{\mathbb{R}}  \left(\frac{x-\mu}{\sigma}\right)^2 \frac{ p_q(x;\xi)^{2q-1}}{\left(-\ell_q(x,\xi)\right)^{j}} dx\\
&=\frac{4}{(3-q)^2}\sum_{j=0}^1   \frac{b_j^2}{(Z_q\sigma)^{2(1-q)}\sigma^2} \Phi(q,2,1,j;\xi),\\
g^{(q,{a})}_{\sigma \sigma}(\xi) 
&=\sum_{j=0}^1 \frac{b_j^2}{(Z_q\sigma)^{2(1-q)}\sigma^2} \int_{\mathbb{R}} 
\left\{1-\left(\frac{x-\mu}{\sigma}\right)^2\right\}^2  \frac{ p_q(x;\xi)^{2q-1}}{\left(-\ell_q(x,\xi)\right)^{j}} dx\qquad\\
&=\sum_{j=0}^1 \frac{b_j^2}{(Z_q\sigma)^{2(1-q)}\sigma^2} \sum_{k=0}^2\binom{2}{k} (-1)^k \Phi(q,2,k , j; \xi)
\end{split}
\end{equation}
for $\xi \in \mathbb{R}\times \Sigma_{q,{a}}$,
where we set
\begin{align*}
\Phi(q,n,k,j;\xi):= \int_{\mathbb{R}}  \left(\frac{x-\mu}{\sigma}\right)^{2k} \frac{ p_q(x;\xi)^{(n-1)(q-1)+q} }{\left(-\ell_q(x,\xi)\right)^{j}} dx.
\end{align*}
\begin{lemma}\label{l7}
For  $n\in \mathbb{N}, k\in \{0, 1,\ldots, n \}$ and  $\xi=(\mu,\sigma)\in\mathbb{R} \times \Sigma_{q,{a}}$,
then
\begin{align*}
&\Phi(q,n,k,0;\xi)\\
&=
\begin{cases}
\displaystyle \frac{\sigma}{(Z_q \sigma)^{(n-1)(q-1)+q}} \left(\frac{3-q}{q-1}\right)^{k+\frac12} B\left(\frac{3-q}{2(q-1)} + n-k, \frac12+k\right)
&\text{if \ }q > 1,\\
(2k-1)!!& \text{if \ }q=1,
\end{cases}
\end{align*}
where by convention $(2\cdot 0-1)!!:=1$.
\end{lemma}
\begin{proof}
We apply the change of variables with 
\[
y=\frac{1}{2}\left(\frac{x-\mu}{\sigma}\right)^2 \ \text{if \ $q=1$, \quad and }
\quad
y=\frac{q-1}{3-q}\left(\frac{x-\mu}{\sigma}\right)^2 \ \text{otherwise}.
\]

For $q=1$, we observe  that 
\begin{align*}
\Phi(1,n,k,0;\xi)&=\int_{\mathbb{R}}   p_1(x;\xi)^{(n-1)(1-1)+1} \left(\frac{x-\mu}{\sigma}\right)^{2k} dx\\
&=
2\int_{0}^{\infty} \frac{1}{\sqrt{2\pi} \sigma }\exp\left(-\frac1{2}\left(\frac{x-\mu}{\sigma}\right)^2 \right)^{(n-1)(1-1)+1} \left(\frac{x-\mu}{\sigma}\right)^{2k}dx\\
&=
\frac{2^k}{\sqrt{\pi}}\int_{0}^{\infty} e^{-y}  y^{k-\frac12}dy\\
&=\frac{2^k}{\sqrt{\pi}}\Gamma\left(k+\frac12\right)
=\frac{2^k}{\sqrt{\pi}}\frac{(2k-1)!!}{2^{k}}\sqrt{\pi}\\
&=(2k-1)!!,
\end{align*}
where $\Gamma(\cdot)$ stands for the Gamma function, that is
\[
\Gamma(s):=\int_0^\infty e^{-x} x^{s-1} dx\qquad \text{for\ }s>0.
\]

For $q>1$, it tuns out that  
\begin{align*}
&\Phi(q,n,k,0;\xi)\\&=\int_{\mathbb{R}}   p_q(x;\xi)^{(n-1)(q-1)+q} \left(\frac{x-\mu}{\sigma}\right)^{2k} dx\\
&=
2\int_{0}^{\infty}\frac{1}{(Z_q \sigma)^{(n-1)(q-1)+q}}\left[1+\frac{q-1}{3-q}\left(\frac{x-\mu}{\sigma}\right)^2 \right]^{\frac{(n-1)(q-1)+q}{1-q}} \left(\frac{x-\mu}{\sigma}\right)^{2k}dx\\
&=
\frac{\sigma}{(Z_q \sigma)^{(n-1)(q-1)+q}} \left(\frac{3-q}{q-1}\right)^{k+\frac12} 
\int_{0}^{\infty}\frac{ y^{k-\frac12}}{(1+y)^{n-1+\frac{q}{q-1} } } dy\\
&=
\frac{\sigma}{(Z_q \sigma)^{(n-1)(q-1)+q}} \left(\frac{3-q}{q-1}\right)^{k+\frac12}
B\left(\frac{3-q}{2(q-1)} + n-k, \frac12+k\right).
\qquad
\qed
\end{align*}
\end{proof}

\begin{proposition}\label{g}
For ${a}=1$ and $\xi=(\mu, \sigma) \in \mathbb{R} \times \Sigma_{q,{a}}$, we have  that
\begin{align*}
g^{(q,1)}_{\mu \mu}(\xi)
=\frac{1}{\sigma^2},\qquad
g^{(q,1)}_{\sigma\sigma}(\xi)
=\frac{3-q}{\sigma^2}.
\end{align*}
\end{proposition}
\begin{proof}
It follows from Lemma \ref{l7} that 
\[
 \Phi(1,2,0,0;\xi)=1, \qquad
 \Phi(1,2,1,0;\xi)=1, \qquad 
 \Phi(1,2,2,0;\xi)=3,
\]
implying
\begin{align*}
g^{(1,1)}_{\mu \mu}(\xi)
=b_0^2(q,1)\frac{1}{\sigma^2}=\frac{1}{\sigma^2},\qquad
g^{(1,1)}_{\sigma \sigma}(\xi) 
=b_0^2(q,1)\sum_{j=0}^1 \frac{1}{\sigma^2} (1-2+3   )=\frac{2}{\sigma^2}.
\end{align*}

Assume $q>1$.
By the property that 
\[
B(s+1,t)=\frac{st}{s+t} B(s,t) \qquad \text{for\ }s,t>0,
\]
we have that 
\begin{align*}
\Phi(q,2,k,0;\xi)
&=
 \frac{\sigma}{(Z_q \sigma)^{(2-1)(q-1)+q}} \left(\frac{3-q}{q-1}\right)^{k+\frac12} B\left(\frac{3-q}{2(q-1)} + 2-k, \frac12+k\right)\\
&=
 \frac{\sigma}{(Z_q \sigma)^{(q-1)+q}} \left(\frac{3-q}{q-1}\right)^{k+\frac12}\frac{f_2(k)}{(\frac{1}{q-1} +1)\cdot \frac{1}{q-1}}  B\left(\frac{3-q}{2(q-1)} , \frac12\right)\\
&=
 \frac{1}{(Z_q \sigma)^{2(q-1)}} \left(\frac{3-q}{q-1}\right)^{k}\frac{(q-1)^2f_2(k)}{q},
 \end{align*}
where we set 
\begin{align*}
f_2(0):&=\left(\frac{3-q}{2(q-1)}+1\right)\cdot \frac{3-q}{2(q-1)}=\frac{(q+1)(3-q)}{4(q-1)^2},\\
f_2(1):&=\frac{3-q}{2(q-1)}\cdot \frac{1}{2}=\frac{3-q}{4(q-1)},\\
f_2(2):&=\frac32\cdot \frac12=\frac34.
\end{align*}
This leads to  that 
\begin{align*}
g^{(q,1)}_{\mu \mu}(\xi)
&=\frac{4}{(3-q)^2}  \frac{b_0^2(q,1)}{(Z_q\sigma)^{2(1-q)}\sigma^2} \Phi(q,2,1,0;\xi)=  \frac{1}{\sigma^2} ,\\
g^{(q,1)}_{\sigma \sigma}(\xi) 
&=\frac{b_0^2(q,1)}{(Z_q\sigma)^{2(1-q)}\sigma^2} \sum_{k=0}^2\binom{2}{k} (-1)^k \Phi(q,2,k , 0; \xi)\\
&=\frac{1}{\sigma^2} \sum_{k=0}^2\binom{2}{k}\left\{ (-1)^k  \left(\frac{3-q}{q-1}\right)^{k} (q-1)^2f_2(k)\right\}
=\frac{3-q}{\sigma^2}.\qquad \qed
\end{align*}
\end{proof}

Fix $n,j\in \mathbb{N}, k\in \{0, 1,\ldots, n \}$ and  $\xi=(\mu,\sigma)\in\mathbb{R} \times \Sigma_{q,{a}}$.
Let us compute $\Phi(q,n,k,j;\xi)$   with the use of the residue theorem.
Note that 
\[
\Phi(q,n,k,j;\mu, \sigma)=\Phi(q,n,k,j;0,\sigma).
\]
Define  a complex valued function $\phi_{q,n,k,j;\sigma}$ on $\mathbb{C}$ by
\begin{align*}
\phi_{q,n,k,j;\sigma}(z)
&:= \left(\frac{z}{\sigma}\right)^{2k} \frac{p_q(z;0,\sigma)^{(n-1)(q-1)+q} }{\left(-\ell_q(z;0,\sigma)\right)^{j}}\\
&=\left(\frac{z}{\sigma}\right)^{2k} {p_q(z;0,\sigma)^{(n-1)(q-1)+q} } \left\{\frac{z^2+r(q,\sigma)^2}{(Z_q\sigma)^{1-q} (3-q)\sigma^2}\right\}^{-j},
\end{align*}
where we set 
\[
r(q,\sigma):=\sqrt{-\ln_q \left(\frac1{Z_q\sigma}\right) \cdot (Z_q \sigma)^{1-q} (3-q)\sigma^2  }.
\]
The function $\phi_{q,n,k,j;\sigma}$  has  poles of order $j$  at $\pm \i r(q,\sigma)$.
For $R>r(q,\sigma)$, let $L_R$ and $C_R$ be smooth curves in $\mathbb{C}$ defined respectively by 
\begin{align*}
L_R:=\{z :[-R, R]\to \mathbb{C}\ |\ z(\theta)=\theta \}, \qquad 
C_R:=\{z:[0,\pi]\to \mathbb{C}\ |\ z(\theta)= R e^{\i \theta} \}.
\end{align*}
The residue theorem yields that
\begin{equation}\label{res}
\int_{L_R\cup C_R}\phi_{q,n,k,j;\sigma}(z)dz
=2\pi\i \cdot \mathrm{Res}( \phi_{q,n,k,j;\sigma}; \i r(q,\sigma)), 
\end{equation}
where 
$\mathrm{Res}( \phi_{q,n,k,j;\sigma}; \i r(q,\sigma))$ stands for the residue of 
$\phi_{q,n,k,j;\sigma}$ at $z=  \i r(q,\sigma)$.

\begin{lemma}\label{resthm}
For  $ n,j\in \mathbb{N}, k\in \{0, 1,\ldots, n \}$ and $(\mu, \sigma)\in \mathbb{R}\times\Sigma_{q,{a}}$, 
then 
\[
\Phi(q,n,k,j;\mu,\sigma)=2\pi\i \cdot \mathrm{Res}( \phi_{q,n,k,j;\sigma}, \i r(q,\sigma)).
\]
\end{lemma}
\begin{proof}
If we show that 
\[
\lim_{R\to \infty} \int_{ C_R}\phi_{q,n,k,j;\sigma}(z)dz=0, 
\]
then we have the desired result by letting $R\to \infty$ in \eqref{res}.

Take $R>r(q,\sigma)$ large enough.
We calculate that 
\begin{align*}
&\left|\int_{ C_R}\phi_{q,n,k,j;\sigma}(z)dz\right| \\
&
\leq
R \int_0^\pi \left|  \phi_{q,n,k,j;\sigma}( R e^{\i \theta})  \right| d\theta \\
&=
R \int_0^\pi   \left(\frac{R}{\sigma} \right)^{2k}   \left| p_q(R e^{\i \theta };0,\sigma) \right|^{(n-1)(q-1)+q} 
\left|\frac{R^2e^{2\i \theta}+r(q,\sigma)^2}{(Z_q\sigma)^{1-q} (3-q)\sigma^2}\right|^{-j} d\theta\\
&\leq 
CR^{2(k-j)+1}
\int_0^\pi \left| \exp_q\left(-\frac{R^2e^{2\i \theta }}{(3-q)\sigma^2}\right) \right|^{(n-1)(q-1)+q}d\theta,  
\end{align*}
where the constant $C$ depends on $q$ and $\sigma$.

In the case $q=1$, we have that 
\begin{align*}
  \left| \exp_1\left(-\frac{R^2e^{2\i \theta }}{(3-1)\sigma^2}\right) \right|^{(n-1)(1-1)+q} 
  =\exp \left(- \frac{R^2 \cos 2\theta}{2\sigma^2}\right),
 \end{align*} 
consequently
\begin{align*}
\left|\int_{ C_R}\phi_{q,n,k,j;\sigma}(z)dz\right|
\leq 
CR^{2(k-j)+1}
\int_0^\pi\exp \left(- \frac{R^2 \cos 2\theta}{2\sigma^2}\right) d\theta
\xrightarrow{R\to\infty}0.
\end{align*}

In the case  $q>1$,  we observe that 
\begin{align*}
  \left|  \exp_q\left(-\frac{R^2e^{2\i \theta }}{(3-q)\sigma^2}\right)\right|^{(n-1)(q-1)+q}
& =\left|  1+\frac{q-1}{3-q} \frac{R^2 e^{2\i\theta}}{\sigma^2} \right|^{\frac{^{(n-1)(q-1)+q}}{1-q}}
\leq  C' R^{-2n+\frac2{1-q}},
\end{align*} 
where  the constant $C'$ depends on $q$ and $\sigma$.
This  yields that 
\begin{align*}
\left|\int_{ C_R}\phi_{q,n,k,j;\sigma}(z)dz\right|
\leq 
C \cdot C'  R^{2(k-j)+1-2n+\frac2{1-q}} \cdot \pi .
\end{align*}
The right-hand side converges to 0 as $R\to \infty$ 
since we have 
\[
2(k-j)+1-2n+\frac2{1-q}\leq -1+\frac2{1-q}<0
\]
due to the assumption  $k\leq n$ and $j\geq 1$.
\qquad
$\qed$
\end{proof}

\begin{proposition}\label{p3}
For $\xi=(\mu, \sigma) \in \mathbb{R} \times \Sigma_{q,{a}}$, then 
\begin{align*}
g^{(q,{a})}_{\mu \mu}(\xi)
&= \frac{b_0^2(q,{a})}{b_0^2(q,1) \sigma^2} -\frac{4}{3-q}\frac{\pi b_1^2(q,{a})}{(Z_q\sigma)^{1-q}\sigma^2} r(q,\sigma),\\
g^{(q,{a})}_{\sigma \sigma}(\xi) 
&=\frac{(3-q)b_0^2(q,{a})}{b_0^2(q,1) \sigma^2} 
+\frac{\pi (3-q) b_1^2(q,1)}{(Z_q\sigma)^{1-q}r(q,\sigma)} \left\{1+\left(\frac{r(q,\sigma)}{\sigma}\right)^2\right\}^2.
\end{align*}
\end{proposition}
\begin{proof}
It follows from Lemma \ref{resthm} that 
\begin{align*}
&\Phi(q,n,k,1;\xi)\\
&=2\pi\i \cdot \mathrm{Res}( \phi_{q,n,k,1;\sigma}, \i r(q,\sigma))\\
&=2\pi\i \lim_{z\to\i r(q,\sigma) } \left\{  \left(z-\i  r(q,\sigma)\right) \cdot \phi_{q,n,k,j;\sigma}(z) \right\}\\
&=2\pi\i \lim_{z\to\i r(q,\sigma) }\left(\frac{z}{\sigma}\right)^{2k} {p_q(z;0,\sigma)^{(n-1)(q-1)+q} } \left\{\frac{z+\i r(q,\sigma)}{(Z_q\sigma)^{1-q} (3-q)\sigma^2}\right\}^{-1}\\
&=2\pi\i \cdot   \left(\frac{\i r(q,\sigma)}{\sigma}\right)^{2k} \frac{(Z_q\sigma)^{1-q}(3-q)\sigma^2}{2\i  r(q,\sigma) } \\
&=(-1)^k \frac{\pi (Z_q\sigma)^{1-q}(3-q)}{\sigma^{2(k-1)}} r(q,\sigma)^{2k-1},
\end{align*}
where we used $p_q(\i r(q,\sigma);0,\sigma )=1$.
This with  Proposition \ref{g} and \eqref{eq:4.1} concludes the proof of the proposition. \qquad $\qed$
\end{proof}

\begin{remark}
In the case $a=1$,  the Riemannian manifold $(\mathcal{S}_{q,1}, g^{(q,1)})$ has a constant curvature $-1/(3-q)$.
This means that all $(\mathcal{S}_{q,1}, g^{(q,1)})$ for $1\leq q<3$ are homothetic to each other.
However, Proposition \ref{p3} suggests that this homothety may fail for $a\neq 1$.
\end{remark}

\section{Concluding remarks}
In this note, we presented gauge freedom of entropies on the subset $\mathcal{S}_q$ of all $q$-Gaussian densities  for $1\leq q<3$. 
%
We showed that  a constant multiple of each $(q, a)$-entropy coincides with the Boltzmann--Shannon entropy if $q=1$,  and  the Tsallis entropy otherwise. 
However,  any constant multiple of the $(q, a)$-relative entropy differs from  the $(q,1)$-relative entropy for $a\neq 1$.
We remark that the $(q, 1)$-relative entropy 
coincides with the Kullback--Leibler divergence if $q=1$, 
and the Tsallis relative entropy of the Csisz\'ar type otherwise. 

In information geometry, 
the Kullback--Leibler divergence projection from observed data to a statistical model
attains the maximum likelihood estimator (see \cite[Chapter 4]{AN-book}). 
The terminology ``maximum" depends on a criterion.
It is known that 
higher-order asymptotic theory of estimation
and Bayesian statistics improve the maximum likelihood estimator in another criterion.
Ishige, Salani and  the second named author  showed  in \cite[Theorem 3.2]{IST-2019} that 
the concavity related to the case $(q,a)=(1,1/2)$ is  the strongest concavity among all admissible  concavities preserved by the heat flow in Euclidean space.
We expect that  the $(1, 1/2)$-relative entropy improves the maximum likelihood estimator.
%
\vspace{-5pt}
\bibliography{bibtex-takatsu2}
\bibliographystyle{splncs04}
\end{document}